\documentclass[11pt,twoside,reqno]{amsart}
\usepackage{amsxtra}
\usepackage{color}
\usepackage{amsopn}
\usepackage{caption} 
\usepackage{amsmath,amsthm,amssymb}
\newtheorem{teo}{Theorem}[section]
\newtheorem{prop}{Proposition}[section]
\newtheorem{corol}{Corollary}[section]
\newtheorem{lemm}{Lemma}[section]
\newtheorem{remark}{Remark}[section]
\newtheorem{ex}{Example}[section]
\newtheorem{definition}{Definition}[section]

\newcommand{\beq}{\begin{equation}}
\newcommand{\eeq}{\end{equation}}
\newcommand{\bqn}{\begin{eqnarray}}
\newcommand{\eqn}{\end{eqnarray}}
\newcommand{\bqne}{\begin{eqnarray*}}
\newcommand{\eqne}{\end{eqnarray*}}

\newcommand{\R}{{\mathbb R}}

\newcommand{\C}{{\mathbb C}}

\title{Locally conformal SKT  almost abelian Lie algebras}

\begin{document}

\author{Louis-Brahim Beaufort}
 \address{ Département de Mathématiques de l'Ecole Normale Supérieure Paris-Saclay, Université Paris-Saclay\\
  4 avenue des Sciences \\
  91190 Gif sur Yvette, France}
\email{louis-brahim.beaufort@ens-paris-saclay.fr}

\author{ Anna Fino}
\address{Dipartimento di Matematica \lq\lq Giuseppe Peano\rq\rq \\ Universit\`a di Torino\\
Via Carlo Alberto 10\\
10123 Torino, Italy\\
$\&$  Department of Mathematics and Statistics\\
Florida International University\\
Miami Florida, 33199, USA}
 \email{annamaria.fino@unito.it, afino@fiu.edu}

\date{\today}

\maketitle

\begin{abstract} 
  A locally conformal SKT  (shortly LCSKT)  structure    is a Hermitian   structure  $(J, g)$ whose   Bismut torsion 3-form $H$ satisfies the condition  $dH = \alpha \wedge H$,  for some closed non-zero 1-form $\alpha$. This condition was introduced in \cite{Ferreira}  as a generalization of the SKT  (or pluriclosed) condition $dH= 0$. In this paper, we  characterize the almost abelian  Lie algebras  admitting a Hermitian structure  $(J, g)$ such that $dH = \alpha \wedge H$,   for some closed 1-form $\alpha$. As an application we classifiy  LCSKT almost abelian Lie algebras in dimension $6$. Finally, we  also study on almost abelian Lie algebras  the compatibility between the LCSKT condition and other types of Hermitian  structures. 
\end{abstract}

\section{Introduction}

Let $(M, J, g)$ be a Hermitian manifold and denote by $\omega$  the \emph{fundamental 2-form}  defined  by 
$\omega(X, Y) = g(JX, Y)$, for any pair of tangent vector fields $X, Y$.  If $(M, J, g)$ is  \emph{K\"ahler},   $J$ and $g$ are both parallel with respect to the Levi-Civita connection $\nabla^{LC}$  of $g$, but  this  is not anymore true in the non-K\"ahler case.

It has been proved by Gauduchon in \cite{gauduchon1997hermitian} that  any  Hermitian manifold  $(M,  J, g)$  admits an affine line of Hermitian connections called \emph{canonical}, preserving both $J$ and $g$ and passing through the  Chern  connection and the Bismut connection. The Bismut connection   $\nabla^B$ can be characterized among these canonical connections as the  only Hermitian  connection with totally skew-symmetric torsion, i.e. such that 
\[ H(X, Y, Z) = g(T^B(X, Y), Z) \]
 is skew-symmetric,  where $T^B(X, Y)$ is the torsion of $\nabla^B$.
In terms of  the fundamental form $\omega$ the $3$-form $H$  has the following expression:
\[H(X, Y, Z) = - d\omega(JX, JY, JZ).\]

A  Hermitian structure $(J, g)$ is called  \emph{Strong Kähler with Torsion} (SKT), or \emph{pluriclosed} if $dH = 0$.  The notion of SKT structure appeared first in theoretical physics \cite{gates1984twisted, howe1988further, strominger1986superstrings}, and is deeply related to generalized Kähler geometry \cite{gates1984twisted, fino2021generalized, gualtieri2014generalized}.  By \cite{Gauduchon84}  every  compact complex surface admits a SKT metric, but  in higher dimension no general conditions are known  for  the existence  of SKT metrics.  Examples of  compact SKT manifolds can be constructed  considering compact quotients  of nilpotent  or solvable Lie groups by lattices (see for instance \cite{MadsenSwann, FPS,  EFV, Arroyo_2019, FOU, fino2021generalized,FP1, FrSw1, FrSw2, AN}). 

A  generalization of  the SKT condition  has been recently  introduced in \cite{Ferreira}:  a Hermitian structure $(J, g)$ is \emph{locally conformally SKT} (LCSKT) if there exists a  closed non-zero   1-form $\alpha$ such that
$dH = \alpha \wedge H.$
Note that   a SKT  structure  is not LCSKT unless there exists a   closed non-zero $1$-form $\alpha$ such that $\alpha \wedge H =0.$ 
In order to include  SKT and LCSKT  structures  as special cases,  we introduce the following
\begin{definition} A Hermitian structure  $(J, g)$   is called {\it twisted SKT}   if    its Bismut torsion 3-form $H$  satisfies $d H = \alpha \wedge H$, for some closed $1$-form $\alpha$.
\end{definition}

We recall that the {\em Lee form}    of a Hermitian  structure $(J, g)$   is the $1$-form $\theta$  defined by $J \star d \star \omega = J d^* \omega$.  A Hermitian metric $g$ is called  \emph{balanced} if $\theta = 0$, and locally conformally balanced (LCB) if $d\theta = 0$.  By \cite{Alexandrov, Popovici}  a SKT structure which is also balanced   has to K\"ahler but  no general result  is known for the LCSKT case.

A classification  of  $6$-dimensional LCSKT  nilpotent Lie algebras   was given in  \cite{Ferreira}, showing that  in contrast to the SKT case there exists a $3$-step nilpotent Lie algebra admitting a LCSKT structure,   but no general classification result  is known for solvable Lie algebras admitting a LCSKT structure. 
In  this paper we  consider  the case of   {\it almost abelian} Lie algebras $\frak g$,  which are defined as   non-abelian  Lie algebras $\mathfrak{g}$ with a codimension  one  abelian ideal $\mathfrak{n}$. In particular,  we study  the existence of  twisted SKT structures,  extending   the results  in  \cite{Arroyo_2019, fino2021generalized, Ferreira}. More results about special  Hermitian metrics on almost abelian Lie algebras have been  recently obtained  in \cite{Yuqin}.

In Section 3 we  obtain  a characterization of  almost abelian Lie algebras admitting a twisted SKT structure  $(J, g)$, showing also a  description of  all the $J$-Hermitian metrics $g$ and   $1$-forms $\alpha$. In   Section 4, we apply the previous characterization  to   classify  $6$-dimensional LCSKT almost abelian Lie algebras.
Finally, in Section 5, we discuss how the LCSKT condition interacts with balanced, locally conformally balanced (LCB) or Bismut-Ricci flat metrics.  In particular, we show that if a balanced and a twisted SKT metric coexist on an almost abelian Lie algebra, then both are deformations of a K\"ahler metric, but that  this is not anymore  true if we replace the balanced condition by the  LCB one.

\subsection*{Acknowledgments}

This paper is the result of a 3 month Master 1 internship of the first author at the University of Torino under the supervision of the second author. The first author thanks her for her hospitality and her guidance.
The second author  is  partially supported by    Project PRIN 2017 \lq \lq Real and complex manifolds: Topology, Geometry and Holomorphic Dynamics”,  by  GNSAGA (Indam) and by a grant from the Simons Foundation (\#944448).
The authors would like to thank  Fabio Paradiso  for  his  helpful comments.

\section{LCSKT structures on almost abelian Lie algebras}  \label{sect3}

We recall that an almost  complex  structure  $J$ on  a Lie algebra  $\mathfrak{g}$ is  given by an endomorphism $J$ of $\mathfrak{g}$ such that $J^2 =  - \mbox{id}$ and  $J$ is integrable if and only if \[N(X, Y) = [JX, JY] - [X, Y] - J[JX, Y] - J[X, JY] = 0.\] A $J$-Hermitian metric  $g$   on the Lie algebra $\frak g$  is a inner product which is  compatible with $J$ and the pair $(J,g)$ is  also called an Hermitian structure on $\frak g$.

Let  $H$ be  the torsion $3$-form of the Bismut connection associated to $(J, g)$.  A   Hermitian metric   $g$  is called {\it twisted SKT}  if $dH = \alpha \wedge H$,   for some  closed 1-form $\alpha \in \mathfrak{g}^*$. In particular,  if  $\alpha =0$, 
   $g$    is   a SKT   (or pluriclosed)   metric and  if $\alpha$ is non-zero,   the metric $g$ is  LCSKT. 
  
Suppose now  that $\mathfrak{g}$  is an  almost abelian Lie algebra of dimension $2n$ and   denote by   $\mathfrak{n}$  be an abelian ideal of $\mathfrak{g}$ of codimension $1$.  Using the description  of Hermitian almost abelian Lie algebras  found in \cite{Arroyo_2019},   we can  extend  the results in \cite{Ferreira} and show  a characterization of  almost abelian Lie algebras  admitting a twisted SKT structure.
More precisely, given  an almost Hermitian structure $(J, g)$ on $\frak g$,   denote by  $\mathfrak{n}_1 = \mathfrak{n} \cap J\mathfrak{n}$ the maximal $J$-invariant subset of $\mathfrak{n}$ and by $\frak k$  the 1-dimensional orthogonal complement of $\frak n$ in $\frak g$ with respect to $g$. One can always find  an  orthonormal basis $(e_1, \ldots , e_{2n})$ of $\frak g$ adapted  to the splitting $\frak g = J \frak k  \oplus \frak n_1  \oplus \frak k$, i.e.  such that
$\frak k = {\text{span}} \langle e_{2n} \rangle,  \frak n_1 = {\mbox{span}} \langle e_2, \ldots ,e_{2n-1} \rangle$ and  $J e_1 = e_{2n}, J e_2 = e_3, \ldots, J e_{2n-2} = e_{2n -1}$. With respect to  the previous  orthonormal adapted  basis, the $(2n - 1) \times (2n - 1)$ matrix $B$ associated with  the endomorphism  $ad_{e_{2n}} \vert_{\frak n}$  is of the form
\begin{equation}\label{ExpressB}
B = 
        \begin{pmatrix}
            a & w^t \\
            v & A
        \end{pmatrix}, 
\end{equation}
 with $a \in \R,  v, w \in \mathfrak{n}_1,  A \in \mathfrak{gl} (n_1).$  The  Lie algebra structure of $\frak g$  is determined by the endomorphism $ad_{e_{2n}} \vert_{\frak n}$ and $\mathfrak g$   can be written as  the semidirect product $\R \ltimes_{B } \frak n$.
Therefore, the almost Hermitian structure $(J, g)$ is  fully  determined by the algebraic data $(a, v, w, A).$  Moreover, 
 by  \cite{Lauret} $(J, g)$  is Hermitian if and only if $w = 0$ and $[A, J_1] = 0$, where  by $J_1$ we  denote the restriction of $J$ to $\frak n_1$.  
 
In the sequel we will denote  by  $ \frak g (a,v,A)$ the  almost abelian Lie algebra  determined  by the algebraic data $(a, v, 0, A)$  and 
  by  $S(M)$ the symmetric part of  a matrix $M$.

In the next remark we will show that, 
given a complex structure $J$ on $\frak g$,  
we can  associate  to $(\frak g, J)$    algebraic data  $(a, v, A)$ such that the endomorphism of $\mathfrak{n}_1$ represented by $A$ only depends on  $J$, up to a non-zero scalar.

   \begin{remark} \label{remdaptedbasis}
 We can show the existence of a $J$-adapted  basis $(e_1, \ldots, e_{2n})$ in the sense  that $\frak n_1 = {\mbox{span}} \langle e_2, \ldots ,e_{2n-1} \rangle$, $\frak n= {\text{span}} \langle e_1, \ldots, e_{2n-1} \rangle$ and    $J e_1 = e_{2n}, J e_2 = e_3, \ldots, J e_{2n-2} = e_{2n -1}$. 
  To construct  such a   basis, one can choose   any  non-zero element $y \in \mathfrak{g}$ such that $y \notin \mathfrak{n}$. Since 
\[ \mbox{span} \langle y, Jy \rangle \cap \mathfrak{n} \neq  \{ 0 \} \]
there exists a nonzero vector  $e_1 \in  \mbox{span}  \langle y, Jy \rangle \cap \mathfrak{n}$.  Let $e_{2n} = J e_1$, then $e_{2n} \notin \mathfrak{n}$, since   otherwise  $\langle e_1, e_{2n} \rangle$ is contained in $\mathfrak{n}_1$, and then $y \in \mathfrak{n}_1$ which is a contradiction.
We  can then complete  $(e_1, e_{2n})$  to a basis of $\frak g$ choosing  a  basis $(e_2, \ldots, e_{2n -1})$  of $\mathfrak{n}_1$ such that $J e_2 = e_3, \ldots, J e_{2n-2} = e_{2n -1}$. 
  
So denoting $\mathfrak{c} = {\mbox{span}} \langle e_{2n} \rangle$ we still have the decomposition $\frak g = J \frak c  \oplus \frak n_1  \oplus \frak c$ and with respect to  the previous $J$-adapted  basis $(e_i)$, the $(2n - 1) \times (2n - 1)$ matrix $B$ associated with  the endomorphism  $ad_{e_{2n}} \vert_{\frak n}$  is of the form \eqref{ExpressB}. Note that,  if we consider another $J$-adapted basis, as $\mathfrak{n}$ is abelian, the endomorphisms $ad_{e_{2n}}\vert_{\mathfrak{n}}$ and $ad_{e_{2n}}\vert_{\mathfrak{n}_1}$ will only change by a non-zero scalar.   So the endomorphisms represented by $A$ and $B$ are actually (up to a non-zero scalar) completely determined by the  almost complex structure $J$.

 One can  now show  that $J$ is integrable if and only if  $B(\mathfrak{n}_1) \subset \mathfrak{n}_1$ and $A = B\vert_{\mathfrak{n}_1}$ commutes with $J_1 = J\vert_{\mathfrak{n}_1}$.
Indeed, if $J$ is integrable, we have 
$$
    [Je_{2n}, Jx] = [e_{2n}, x] + J[Je_{2n}, x] + J[e_{2n}, Jx], \quad \forall x \in \mathfrak{n}_1,
    $$
    and so
    $$
    [e_1,  Jx] = Bx + J [e_1, x] + JB(Jx).
    $$
   Since  $\frak n$ is abelian and $\mathfrak{n}_1$  is J-invariant, it follows that  $[e_1,  Jx] =  [e_1,  Jx] = 0$ and so  $JBJx = -Bx$.
As a  consequence  $B\mathfrak{n}_1$ is $J$-invariant and  contained in $\mathfrak{n}$, so $B\mathfrak{n}_1 \subset \mathfrak{n}_1$ and $[A, J_1]=0$. In particular, writing $B$  with respect to  a $J$-adapted basis
\begin{equation*}
B = 
        \begin{pmatrix}
            a & w^t \\
            v & A
        \end{pmatrix}, 
\end{equation*}we see that $w = 0$. Conversely, if $B\mathfrak{n}_1 \subset \mathfrak{n}_1$ and $A$ and $J_1$ commute, one  can  check  by a direct computation that  the Nijenhuis tensor vanishes.

Therefore, summarizing the previous discussion, given any $J$-adapted basis $(e_i)$ of an almost abelian Lie algebra $\frak g$ endowed with a complex structure $J$,  we have that $\mathfrak{g} = \frak g (a, v, A)$  and so the Lie algebra  is completely determined by the data  $(a, v, A)$.

Moreover,  given a $J$-Hermitian  inner product $g$ on $\frak g$,  we  can  find an \emph{orthonormal} $J$-adapted basis $(e_i)$  of $\mathfrak{g}$,  choosing  as  $e_{2n}$ a generator of the orthogonal complement $\mathfrak{n}^{\perp_{g}}$ of $\frak n$.
\end{remark}

In \cite[Lemma 1]{Arroyo_2019} it was proved that a Hermitian almost abelian Lie algebra $(\frak g (a,v,A), J, g)$ is SKT  (i.e. $d H =0$) if and only  if $$a A + A^2 + A^T  A \in \frak{so}({\frak n}_1).$$ Moreover, 
$(\frak g (a,v,A), J, g)$ is K\"ahler  (i.e. $H =0$) if and only if $v = 0$ and $A$ is antisymmetric
(see for instance \cite[Lemma 3.6]{fino2021generalized}).

The following result, valid for LCSKT structures, is actually valid for every twisted SKT structure $(J, g)$:
 
 \begin{prop}[{\cite[Proposition 4.1]{Ferreira}}] \label{lem:relLCSKT} A Hermitian  almost abelian Lie algebra  $(\frak g (a,v,A), J,$ $ g)$ is LCSKT  if and only if, for some nonzero closed 1-form $\alpha$,  the following conditions hold:
\begin{align}
    &\alpha(a e_1 + v) = 0, \label{eq:11}\\
    &\alpha \circ A = 0, \label{eq:12}\\
    &\alpha(X)g(S(A)J_1Y, Z) - \alpha(Y)g(S(A)J_1X, Z) + \alpha(Z)g(S(A)J_1Y, X) = 0,   \label{eq:13}\\
    &g(S((a + \alpha(e_{2n}))A + A^2 + A^T A)J_1Y, Z) = \frac{1}{2} \left ( g(v, Y) \alpha(Z) - g(v, Z) \alpha(Y)\right ), \label{eq:14}
\end{align}
for every $X, Y, Z \in \mathfrak{n}.$
\end{prop}

The previous  result can be extended   in the following 

\begin{prop}  \label{pps:antisym} If a Hermitian  almost abelian Lie algebra  $(\frak g (a,v,A), J, g)$  is twisted SKT, for some closed $1$-form $\alpha$, then
\begin{equation} \label{cond-pps:antisym}
(a + \alpha(e_{2n}))A + A^2 + A^T A \in \mathfrak{so(n_1)}.
\end{equation}
\end{prop}

\begin{proof}
By the proof of Proposition   \ref{lem:relLCSKT}   we know that, for some closed 1-form  $\alpha$ the conditions   \eqref{eq:11},  \eqref{eq:12}, \eqref{eq:13} and  \eqref{eq:14}
are  satisfied.   Note that the condition  \eqref{cond-pps:antisym} is equivalent to the vanishing of the matrix   $C := S((a + \alpha(e_{2n}))A + A^2 + A^T A)$.  To  prove  $C =0$ we  will study   separately the two cases: $v =0$ and $v \neq 0$.

 If $v = 0$, the result  follows   from \eqref{eq:14}.  If $v \neq 0$, we  will  prove  that $C=0$ as a consequence of the following properties:
\begin{align} 
&C = 0 \Longleftrightarrow g(Cv, v) = 0, \label{lem:Czero} \\
 & v \perp {\rm Im}  (A),  \label{lem:vperpImA}\\
 &  v \perp  {\rm {Im}} (A^T).\label{lem:vperpAT}
 \end{align}

Let's first prove    \eqref{lem:Czero}. Using   \eqref{eq:14} we get
\begin{equation}\label{new(4)}
 g(CJ_1Y, Z) = \frac{1}{2} \left ( g(v, Y) \alpha(Z) - g(v, Z) \alpha(Y)\right ), \quad \forall Y, Z \in \frak n.
\end{equation}
For $Y=v,$ we obtain 
  $$
 g(CJ_1v, Z) = \frac{1}{2} \left ( \Vert v \Vert^2 \alpha(Z) - g(v, Z) \alpha(v)\right ),  \quad \forall Z \in \frak n,
 $$
which gives
$$
   \alpha(Z) = \frac{1}{\Vert v \Vert^2} g(2CJ_1 v + \alpha(v)v, Z), \quad  \forall Z \in \frak n.
   $$
 If we   substitute the  previous expression of $\alpha (Z)$  in \eqref{new(4)} we have:
 $$
 \begin{array}{cll}
            g(CJ_1Y, Z) &=& \frac{1}{2} \left ( g \left ( \frac{v}{\Vert v \Vert^2}, Y \right ) g \left (2CJ_1 v + \alpha(v)v, Z \right) - g \left (\frac{v}{\Vert v \Vert^2}, Z \right ) g(2CJ_1 v + \alpha(v)v, Y)\right ) \\
            &= &g \left (\frac{v}{\Vert v \Vert^2}, Y \right ) g(CJ_1 v, Z) - g \left (\frac{v}{\Vert v \Vert^2}, Z\right ) g(CJ_1 v, Y),
\end{array}
$$
for every $Y, Z \in \frak n$.
Let $\pi(Z) := g(\frac{v}{\Vert v \Vert^2}, Z)$. Since $CJ_1$ is also  antisymmetric, 
we get
        \begin{align*}
            g(CJ_1Y, Z) &= \pi(Y) g(CJ_1 v, Z) - \pi(Z) g(CJ_1v, Y) \\
            &= \Vert v \Vert^2 \left (- \pi(Y) \pi(CJ_1 Z) + \pi(Z) \pi(CJ_1 Y) \right ),  \quad \forall Y, Z \in \frak n.
        \end{align*}
        Thus,
        \[g(CY, Z) = \Vert v \Vert^2 \left (\pi(CY) \pi(Z) + \pi(J_1Y) \pi(CJ_1Z) \right ), \quad \forall Y, Z \in \frak n.\]
        Now, since  $g(CJ_1v, v) = 0$,  we obtain $$g(v, CZ) = g(Cv, Z) = g(Cv, v) \pi(Z),  \quad \forall Z \in \frak n$$ and
        \[g(CY, Z) = g(Cv, v) \left (\pi(Y) \pi(Z) + \pi(J_1Y) \pi(J_1Z) \right ).\]
        In particular,    \eqref{lem:Czero} holds.
        
    Using the previous results, we  will now   prove   \eqref{lem:vperpImA}.
  By
        \begin{align*}
            \alpha(Z) &= \frac{1}{\Vert v \Vert^2} g(2CJ_1 v + \alpha(v)v, Z)  = \alpha(v) \pi(Z) - 2 \pi(CJ_1Z) \\
                       &= \alpha(v) \pi(Z) - 2 \pi(Cv) \pi(J_1Z) = \frac{1}{\Vert v \Vert^2} g(\alpha(v)v + 2\pi(Cv)J_1v, Z) 
                               \end{align*}
       it follows that  \eqref{eq:12} is equivalent to $\alpha(v)v + 2 \pi(Cv)J_1v \perp {\rm {Im}} (A)$. But then since $A$ and $J_1$ commute and $J$ is an isometry, we have that $\alpha(v)J_1v - 2 \pi(Cv)v$ is orthogonal to   ${\rm {Im}}  (J_1A )=  {\rm {Im}}  (AJ_1) =  {\rm {Im}}  (A)$. As a consequence both $v$ and $Jv$ are perpendicular to  ${\rm {Im}}  (A)$, or equivalently  $\alpha(v) = \pi(Cv) = 0$.  Then  $\pi(Cv) = 0 \Leftrightarrow C = 0$ and  \eqref{lem:vperpImA}    follows.

To prove \eqref{lem:vperpAT},  we can use 
 the condition \eqref{eq:13}:
        \[\alpha(X)g(S(A)J_1Y, Z) - \alpha(Y)g(S(A)J_1X, Z) + \alpha(Z)g(S(A)J_1Y, X) = 0, \quad \forall Y, Z \in \frak n.\]
        Since $\alpha (Y) = \alpha(Z) =0$  if  $Y, Z$  do not belong to the  span of $v, J_1v$, the previous  condition is equivalent to
        \begin{align*}
            \alpha(v) g(S(A)J_1^2v, Z) - \alpha(J_1v)g(S(A)J_1v, Z) + \alpha(Z) g(S(A)J_1^2v, v) &= 0, \\
            - \alpha(v) \pi(S(A)Z) + \alpha(J_1v) \pi(S(A)Z) - \alpha(Z) \pi(S(A)v) &= 0, \\
            - \alpha(v) \pi(S(A)Z) + 2\pi(Cv) \pi(S(A)J_1Z) &= \alpha(Z) \pi(S(A)v), \quad \forall Z \in \frak n.
        \end{align*}
        Now, as $\pi(S(A)v) = \frac{1}{\Vert v \Vert ^2} g(Av, v) = 0$ and  $v \perp  {\rm {Im}} (A)$,  we get the two conditions 
        \begin{align*}
            - \alpha(v) \pi(S(A)Z) + 2\pi(Cv) \pi(S(A)J_1Z) &= 0, \\
            - g(-\alpha(v)v + 2 \pi(Cv) J_1v, S(A) Z) &= 0, \quad \forall Z \in \frak n,
        \end{align*}
       which imply   that $\alpha(v)v + 2 \pi(Cv) J_1v$ is orthogonal to ${\rm {Im}} \,  S(A)$. But  $\alpha(v)v + 2 \pi(Cv) J_1v$ is also orthogonal to ${\rm {Im}} (A)$,  then  $\alpha(v)v + 2 \pi(Cv) J_1v \perp {\rm {Im}}  (A^T)$, and we may conclude as  before that $v \perp  {\rm {Im}}  (A^T)$ using that $[J_1, A^T] = 0$.

By \eqref{lem:vperpImA} and \eqref{lem:vperpAT} we obtain that $v \in ( {\rm {Im}} (A))^\perp \cap ({\rm {Im}}  (A^T))^\perp = \ker  (A^T) \cap \ker (A)$ and so $Cv = 0$ by definition of $C$. Then, $g(Cv, v) = 0$ and by \eqref{lem:Czero}   it follows that  $C = 0$.

\end{proof}

In   \cite{Ferreira} it was also  proved that  if $( \frak g (a,v,A), J, g)$ is a  Hermitian  almost abelian Lie algebra  with   $A$  invertible, then $(J, g)$ is LCSKT with 1-form $\alpha$ if and only if  the matrix $A$ is $g$-normal and
 $Re(\lambda) \in \{0, - \frac{a+\alpha(e_{2n})}{2} \}$, for every eigenvalue  $\lambda \in {\rm {Spec}} (A)$. This statement also generalizes  with the same proof to Hermitian structures such that $dH = \alpha \wedge H$ for some closed 1-form $\alpha$.

\smallskip

In view of Remark \ref{remdaptedbasis} we can extend the previous result, using the fact that,   given an almost abelian Lie algebra $\frak g$ with a complex structure $J$,  
the endomorphism $ad_{e_{2n}}\vert_{\mathfrak{n}_1}$ is up to a scalar independent from the choice of a $J$-adapted basis.  
As a consequence,  we can prove the following

\begin{teo} \label{thm:caractLCSKT} 
Let $\mathfrak{g}$ be an almost abelian Lie algebra endowed with a complex structure $J$.  Then $\frak g$  admits a twisted SKT $J$-Hermitian  metric if and only if   
$ad_{e_{2n}}\vert_{\mathfrak{n_1}}$ satisfies the following two conditions:

\begin{enumerate}

\item[(i)]  $ad_{e_{2n}}\vert_{\mathfrak{n_1}}$ is  diagonalizable over $\C$ and 

\item[(ii)]    there exist $\mu \in \R$  such that  $Re(\lambda) \in \{0, \mu\}$,  for every  $\lambda \in  {\rm {Spec}} (ad_{e_{2n}}\vert_{\mathfrak{n_1}}).$
\end{enumerate}
Moreover, if $ad_{e_{2n}}\vert_{\mathfrak{n_1}}$ satisfies  $(i)$ and $(ii)$ then  the twisted SKT  $J$-Hermitian metrics are all  inner products $g'$ on $\frak g$  compatible with $J$ and such that $ad_{e_{2n}}\vert_{\mathfrak{n_1}}$ is $g'$-normal. \end{teo}

\begin{proof} Suppose  first that   $(\mathfrak{g},J)$ admits a  $J$-Hermitian metric  $g$ such that $dH = \alpha \wedge H$ for some closed 1-form $\alpha$. Then  we know that there exists  an  orthonormal   $J$-adapted basis  $(e_i)$   of $\mathfrak{g}$ such that $\mathfrak{n}^\perp = \mbox{span} \langle e_{2n} \rangle$, $e_1 = - Je_{2n}$ spans $\mathfrak{n}_1^\perp \cap \mathfrak{n}$,  $\frak n_1 = {\mbox{span}} \langle e_2,  \ldots,  e_{2n-1}\rangle$ and $J e_2 = e_3, \ldots, J e_{2n-2} = e_{2n -1}$.  With respect to  the previous  basis,  the endomorphism  $ad_{e_{2n}} \vert_{\frak n}$ can be written as
    $B = 
        \begin{pmatrix}
            a & 0 \\
            v & A
        \end{pmatrix}$ with $a \in \R, v \in \mathfrak{n}_1, A \in \mathfrak{gl(n_1)}$.  By Proposition  \ref{pps:antisym} it follows  that  $(a + \alpha(e_{2n}))A + A^2 + A^T A \in \mathfrak{so(n_1)}.$   To prove that $A$ is diagonalizable    and that  the real part of the eigenvalues of $A$ are either 0 or $- \frac{1}{2}(a + \alpha(e_{2n}))$
 we  can proceed in a similar way as in   \cite[Lemma 4.2, Theorem 4.3]{Ferreira} (see also \cite{Arroyo_2019}). We consider the complexification $\mathfrak{n}_1^{\C}$, to which we extend $A$ by linearity and $g$ as an Hermitian inner product. Then, if $\lambda$ is an eigenvalue of $A$ with eigenvector $z$, Proposition \ref{pps:antisym} gives
\[0 = g(((a + \alpha(e_{2n}))A + A^2 + A^T A)z, \bar{z}) = \lambda(a + \alpha(e_{2n}) + 2\lambda) g(z, \bar{z}), \quad  \forall z \in \mathfrak{n}_1^{\C}. \]
Therefore, if $\lambda \notin \{ 0, - \frac{1}{2}(a + \alpha(e_{2n}))\}$, it follows that $g(z, \bar{z}) = 0$. Writing $z = x + iy$, this means that $\Vert x \Vert = \Vert y \Vert$ and $g(x, y) = 0$. If $\lambda = \mu + i \nu, Ax = \mu x - \nu y, Ay = \mu y + \nu x$, then applying Proposition \ref{pps:antisym} again gives us 
\[0 = g(((a + \alpha(e_{2n}))A + A^2 + A^T A)x, x) = \mu(a + \alpha(e_{2n}) + 2\mu). \]
As a consequence, $\mu = Re(\lambda) \in \{0, - \frac{1}{2}(a + \alpha(e_{2n})) \}$ and  we  can use the following general  linear algebra estimate:
if $M \in \mathfrak{gl}_m(\mathbb{C})$, then 
\[\Vert S(M) \Vert^2 \geq \sum_{\lambda \in {\rm {Spec}} (M)} Re(\lambda)^2\]
with equality if and only if $M$ is a normal matrix  (see \cite[Corollary A.2]{Arroyo_2019}).

Since $J_1$ commutes with $A$, the eigenvalues of $A$ come in pairs and can be arranged such that
\[\lambda_1 =  \ldots  = \lambda_{2k} = - \frac{1}{2}(a + \alpha(e_{2n})), \lambda_{2k+1} = ... = \lambda_{2n-2} = 0, \]
for some integer $k$. Taking traces in Proposition \ref{pps:antisym}  we obtain 
\[\Vert S(A) \Vert^2 = \frac{1}{2} tr(A^2 + A^T A) = - \frac{1}{2} (a + \alpha(e_{2n})) tr(A) = + \frac{1}{2} k (a + \alpha(e_{2n}))^2 = \sum_{\lambda \in {\rm {Spec}}(A)} Re(\lambda)^2. \]
Thus by  the previous estimate  $A$ is normal and so it is  in particular diagonalizable.

We will now prove the converse. Fixed  a $J$-adapted basis, assume  that  $\frak g = \frak g (a,v,A)$  with $A$  diagonalizable and  such that  $Re(\lambda) \in \{0, \mu\}$,  for every  $\lambda \in  {\rm {Spec}} (A)$ for some $\mu \in \R$.  Note that these two conditions  on $A$ do not depend on the choice of the $J$-adapted basis.
We  will  show that  there  exists a  $J$-Hermitian metric  $g$ with respect to which $A$ is normal. 

We first complexify $\mathfrak{n}_1$ to obtain an eigenbasis $\varepsilon$  for $A$.  We can  then prove that  there exists an  hermitian inner product on $\frak n_1$  such that $\varepsilon$ is orthonormal. Indeed,  if we identify $\mathfrak{n}_1 \otimes_{\mathbb{R}} \mathbb{C}$ with $\mathbb{C}^{2n-2}$ and consider the transition matrix $P$ from the canonical basis  of $\mathbb C^{n -2}$ to $\varepsilon$,  we can consider the bilinear form $h$ whose associated matrix with respect to  the canonical basis is given by $\overline{(P^{-1})^T}P^{-1}$.  This matrix is self-adjoint and positive-definite, so $h$ is an inner product, and since the canonical basis is orthonormal for the canonical inner product on $\mathbb{C}^{2n-2}$, it follows that $\varepsilon$ is orthonormal  with respect to $h$.  By definition, $A$ is diagonal in an orthonormal basis for $h$, so $A$  normal with respec to $h$.
To  determine a   $J$-Hermitian metric  $g$   on $(\frak g, J)$ with respect to which $A$ is normal,   we first  define $g$ on $\mathfrak{n}_1$ as    \[ g(x,y) = \frac{1}{2} Re(h(x, y) + h(Jx, Jy)), \quad \forall x,y \in \frak n_1. \] By construction, $J$ is an isometry of  $(\frak n_1, g)$. 

Let $A^*$ be the $h$-adjoint of $A$, then since $A$ is $h$-normal, $A^*$ is a polynomial in $A$. To see this, first diagonalize simultaneously $A$ and $A^*$, then it is easy to see that any formal polynomial $Q$ sending each eigenvalue of $A$ on its conjugate (e.g. the corresponding Lagrange polynomial) satisfies $A^* = Q(A)$.  Since  $J$  is integrable, it follows that $A^*$ commutes with $J_1$, and we have
\begin{align*}
    g(Ax, y) &= \frac{1}{2} Re \left ( h(x, Ay) + h(J_1Ax, J_1y)\right ) \\
    &= \frac{1}{2} Re \left ( h(A^*x, y) + h(AJ_1x, J_1y)\right ) \\
    &= \frac{1}{2} Re \left ( h(A^*x, y) + h(J_1x, A^*J_1y)\right ) \\
    &= \frac{1}{2} Re \left ( h(A^*x, y) + h(J_1x, J_1A^*y)\right ) \\
    &=g(x, A^*y), \quad \forall x, y \in \frak n_1.
\end{align*}
 As a consequence, $A$ is $g$-normal, and we may extend $g$ to $\mathfrak{g}$ arbitrarily (and $J$-orthogonally) to obtain the desired metric.

We will now construct a  closed  $1$-form $\alpha$ such that $dH = \alpha \wedge H$, i.e.  satisfying the  four conditions in Proposition \ref{lem:relLCSKT}. To do this,  we first assume that such an $\alpha$ exists and then we  will derive  the conditions that  this hypothetical $\alpha$  has to satisfy. These conditions will allow us  to construct $\alpha$ and also  to determine all possible $\alpha$. We  know that there exists a $g$-orthonormal  $J$-adapted basis   $(e_i)$  of  $\mathfrak{g}$ such that $\mathfrak{n}^\perp = \mbox{span} \langle e_{2n} \rangle$,  $e_1 = - Je_{2n}$ spans $\mathfrak{n}_1^\perp \cap \mathfrak{n}$ and $\mathfrak{n}_1 = \mbox{span}  \langle  e_2,  \ldots,  e_{2n-1} \rangle$. With respect to this basis  we  have $B = 
    \begin{pmatrix}
        a & 0 \\
        v & A
    \end{pmatrix}$ with $a \in \R, v \in \mathfrak{n}_1, A \in \mathfrak{gl(n_1)}$.

We can show that  either $\alpha \vert_{\mathfrak{n}_1} = 0$,  or $A$ is antisymmetric.
Indeed, let $(\varepsilon'_k)$ be an unitary basis of $\mathfrak{n}_1 \otimes_{\R} \C$ in which both $A$ and $J_1$ are diagonal (it exists since  $A$ and $J_1$  commute and are both $g$-normal). Let $$A\varepsilon'_l = \lambda_l\varepsilon'_l, J\varepsilon'_l = \nu_l \varepsilon'_l, \alpha_l = \alpha(\varepsilon'_l).$$ Then \eqref{eq:13} is satisfied if and only if
\[\alpha_j \lambda_k \nu_k \delta_{k,l} + \alpha_j \bar{\lambda_l} \nu_k \delta_{k,l} - \alpha_k \lambda_j \nu_j \delta_{j,l} - \alpha_k \bar{\lambda_l} \nu_j \delta_{j,l} + \alpha_l \lambda_k \nu_k \delta_{k,j} + \alpha_l \bar{\lambda_j} \nu_k \delta_{k,j} = 0\]
where by  $\delta_{j, k}$  we denote  the Kronecker symbol.
When $j=k$, we obtain $\alpha_l Re(\lambda_k) = 0$, so either $\alpha$ is zero on $\mathfrak{n}_1$, or all of the eigenvalues of $A$ are purely imaginary, which is equivalent, by the spectral theorem,  to $A$ being antisymmetric as $A$ is normal.

We can now finish  the proof  studying separately the two cases $\alpha \vert_{\mathfrak{n}_1} = 0$ and $A$ antisymmetric, constructing in both cases an explicit closed $1$-form $\alpha$ such that $d H = \alpha \wedge H$. 

 If  $\alpha \vert_{\mathfrak{n}_1} = 0$, we can prooced  as in \cite{Ferreira}.
Since  $\alpha$ vanishes  on $\mathfrak{n}_1$,  by \eqref{eq:11} we obtain    $a \, \alpha(e_1) = 0$, \eqref{eq:12} and \eqref{eq:13} are trivially satisfied, and \eqref{eq:14} gives us the condition $$(a + \alpha(e_{2n}))A + A^2 + A^T A \in \mathfrak{so(n_1)}.$$ 
In a simultaneous eigenbasis, we see that the previous condition is equivalent to 
\[(a + \alpha(e_{2n}))\lambda + \lambda^2 + \vert \lambda \vert^2 \in i \R,  \] for every $\lambda \in {\rm {Spec}} (A)$, and taking real parts gives 
\[(a + \alpha(e_{2n}) + 2 Re(\lambda)) Re(\lambda) = 0. \]
Since  $Re(\lambda) \in \{0, \mu\}$, we get the condition $\alpha(e_{2n}) = -a - 2\mu$ if $A$ is not antisymmetric, and no condition else.

In the case $A \in \mathfrak{so(n_1)}$,  since $S(A) = 0$, \ the condition \eqref{eq:13} is trivially satisfied. Since the matrix  $$(a + \alpha(e_{2n}))A + A^2 + A^T A = (a + \alpha(e_{2n}))A$$ is also antisymmetric we have that $ S(a + \alpha(e_{2n}))A + A^2 + A^T A )=0$. Therefore,   \eqref{eq:14} is equivalent to $\alpha(Y)g(v, Z) = \alpha(Z)g(v, Y),$ for every $Y, Z \in \mathfrak{n}_1$. Then, either $v = 0$, or 
\[\alpha(Z) = \alpha(v)  \, g \left ( \frac{v}{\Vert v \Vert^2}, Z \right ), \quad  \forall Z \in \frak n_1.\]
So we get  the following    additional conditions:
\begin{enumerate} 
\item[ i) ] $
a  \, \alpha(e_1) + \alpha(v) = 0,
$
or

\item[ii)]  $\begin{cases}
    
    \alpha \circ A = 0 \text{ if } v = 0, \\[0.5ex]
    
    v \perp {\rm {Im}}  A \text{ and } \alpha(Z) = \alpha(v) g(\frac{v}{\Vert v \Vert^2}, Z) \text{ if } v \neq 0, \quad \forall Z \in \frak n_1,
    \end{cases}$
    \end{enumerate} 
   For  all the cases we can choose $\alpha$ proportional to $e^{2n}$. 

As a consequence, in both cases $\alpha \vert_{\mathfrak{n}_1} = 0$ and $A$ antisymmetric, we can construct  a  closed $1$-form $\alpha$ and determine also all possible $\alpha$ (see Table \ref{fig:alphas}).

Finally, we  can prove that  any other   $J$-Hermitian metric  $g'$   satisfies   the condition  $dH' = \alpha \wedge H'$,  for some closed 1-form $\alpha$,  if and only if  $g'$ is such that $A$ is $g'$-normal.  To see this, we consider a $g'$-orthonormal   $J$-adapted basis  $(e'_i)$ and we write   the endomorphism  $ad_{e'_{2n}} \vert_{\frak n}$ as
    $
        \begin{pmatrix}
            a' & 0 \\
            v' & A'
        \end{pmatrix}$ with $a' \in \R, v' \in \mathfrak{n}_1$ and $A' \in \mathfrak{gl(n_1)}$. Note that $A' $ is  proportional to $A$.
By the  first part of the proof   if   $dH' = \alpha \wedge H'$ for some closed $1$-form, then    $A$ has to be $g'$-normal, and conversely if  $A$ is $g'$-normal, then the last part of the proof tells us that we can find  a closed $1$-form $\alpha$ such that $dH' = \alpha \wedge H'$.

\end{proof}

\begin{remark}  Note that,  if we choose a $J$-adapted basis $(e_i)$,  the previous proof also  provides for each Hermitian metric obtained from Theorem 2.1 the set of closed 1-forms $\alpha$ such that $dH = \alpha \wedge H$ (see Table \ref{fig:alphas}). 
Moreover, it is easy to determine  which $J$-Hermitian metrics  are in fact SKT (i.e. satisfy $dH = 0)$:  the  affine space of closed $1$-forms $\alpha$ such that $dH  = \alpha \wedge H$ has to coincide with   the vector space $\ker(H) = \{\alpha, \alpha \wedge H = 0\}$. 
\end{remark}

As a consequence  of Theorem  \ref{thm:caractLCSKT}  we obtain the following  characterization of LCSKT almost abelian Lie algebras:

\begin{corol} \label{cor-LSKT}
Let $\mathfrak{g}$ be an almost abelian Lie algebra endowed with a complex structure $J$.  Then $\frak g$  admits a LCSKT  metric $g$ if and only if  in any $J$-adapted basis, $\frak g = \frak g (a,v,A)$  with $A$  satisfying the following two conditions:

\begin{enumerate}

\item[(i)]  $A$ is  diagonalizable over $\C$, 

\item[(ii)] either $Re(\lambda) = 0$ for every  $\lambda \in  {\rm {Spec}} (A)$ or   there exists $\mu \in \R \setminus  \{  -\frac{a}{2} \}$  such that  $Re(\lambda) \in \{0, \mu\}$,  for every  $\lambda \in  {\rm {Spec}} (A).$

\end{enumerate}
Moreover, if  $\frak g = \frak g (a,v,A)$, with $A$ satisfying  $(i)$ and $(ii)$,  the LCSKT metrics are all the inner products compatible with $J$ and such that $A$ is $g'$-normal.
\end{corol}

\section{Classification in  dimension six}  \label{sect4}

 If   $\mathfrak{g}$  is a $4$-dimensional  almost abelian Lie algebra with a complex structure $J$, as a consequence of  Corollary \ref{cor-LSKT}  one can  easily show that there exists  a twisted SKT $J$-Hermitian metric $g$. Indeed,  the matrix  $A$ is of  order  $2$ and commutes with $J$, so   it  is diagonalizable and ${\mbox {Re}}  ({\rm {Spec}} (A))$ is reduced to a single element. 

\smallskip

For $6$-dimensional almost abelian Lie algebras
we recall the following characterization for the existence of a complex structure  in terms of  $J$-adapted bases.

\begin{teo}[{\cite[Theorem 3.2]{fino2021generalized}}] \label{thm:class6cplx}
Let $\mathfrak{g}$ be a $6$-dimensional almost abelian Lie algebra. Then $\mathfrak{g}$ admits a complex structure $J$  if and only if the matrix $A$ associated to  $ad_{e_6}\vert_{\mathfrak{n}_1}$,  with respect to a $J$-adapted basis $(e_i)$, is   one of the following  (remember that  $ad_{e_{6}}$ is defined up to a scalar):
\[
    \begin{pmatrix}
        p & 0 & 0 & 0 \\
        0 & q & 0 & 0 \\
        0 & 0 & r & 0 \\
        0 & 0 & 0 & s
    \end{pmatrix},
    \begin{pmatrix}
        p & 1 & 0 & 0 \\
        -1 & p & 0 & 0 \\
        0 & 0 & q & 0 \\
        0 & 0 & 0 & r
    \end{pmatrix},
    \begin{pmatrix}
        p & 1 & 0 & 0 \\
        -1 & p & 0 & 0 \\
        0 & 0 & q & r \\
        0 & 0 & -r & q
    \end{pmatrix},
    \]
\[
    \begin{pmatrix}
        p & 1 & 0 & 0 \\
        0 & p & 0 & 0 \\
        0 & 0 & p & 1 \\
        0 & 0 & 0 & p
    \end{pmatrix},
    \begin{pmatrix}
        p & 1 & -1 & 0 \\
        -1 & p & 0 & -1 \\
        0 & 0 & p & 1 \\
        0 & 0 & -1 & p
    \end{pmatrix}.
\]
\end{teo}

In \cite{fino2021generalized}  all the endomorphisms $ad_{e_{6}}$ admitting such $A$, were determined, obtaining in this way  a classification  of $6$-dimensional  almost abelian Lie algebras admitting a complex structure
 and also  a  classification   of  $6$-dimensional SKT  almost abelian  Lie algebras.
Combining    Theorem  \ref{thm:class6cplx} with   Theorem \ref{thm:caractLCSKT}  we can prove the following

\begin{corol} \label{csq:class6dHaH}

Let $\mathfrak{g}$ be a $6$-dimensional almost abelian Lie algebra  equipped with a complex structure $J$. Then $\mathfrak{g}$ admits a  twisted SKT  metric if and only if the matrix $A$ of  $ad_{e_{6}}\vert_{\mathfrak{n_1}}$,  with respect to a $J$-adapted basis,  has one of the following  expressions: \[
    \begin{pmatrix}
        p & 0 & 0 & 0 \\
        0 & p & 0 & 0 \\
        0 & 0 & p & 0 \\
        0 & 0 & 0 & p
    \end{pmatrix},
    \begin{pmatrix}
        p & 0 & 0 & 0 \\
        0 & p & 0 & 0 \\
        0 & 0 & 0 & 0 \\
        0 & 0 & 0 & 0
    \end{pmatrix},
\]
\[
    \begin{pmatrix}
        p & 1 & 0 & 0 \\
        -1 & p & 0 & 0 \\
        0 & 0 & p & 0 \\
        0 & 0 & 0 & p
    \end{pmatrix},
    \begin{pmatrix}
        0 & 1 & 0 & 0 \\
        -1 & 0 & 0 & 0 \\
        0 & 0 & q & 0 \\
        0 & 0 & 0 & q
    \end{pmatrix},
    \begin{pmatrix}
        p & 1 & 0 & 0 \\
        -1 & p & 0 & 0 \\
        0 & 0 & 0 & 0 \\
        0 & 0 & 0 & 0
    \end{pmatrix},
\]
\[
    \begin{pmatrix}
        p & r & 0 & 0 \\
        -r & p & 0 & 0 \\
        0 & 0 & p & s \\
        0 & 0 & -s & p
    \end{pmatrix},
    \begin{pmatrix}
        p & r & 0 & 0 \\
        -r & p & 0 & 0 \\
        0 & 0 & 0 & s \\
        0 & 0 & -s & 0
    \end{pmatrix}.
\]
In the notations of \cite{fino2021generalized}, the corresponding Lie algebras are:
$\mathfrak{l_1}^{p, p}$, $\mathfrak{l}_1^{p, 0} \cong \mathfrak{l}_1^{0, p} \cong \mathfrak{l}_{17}^p$, $\mathfrak{l}_2^{q, 0} \cong \mathfrak{l}_3^{0, q} \cong \mathfrak{l}_{18}^q$,$\mathfrak{l}_2^{q, 1} \cong \mathfrak{l}_3^{1, q}$, $\mathfrak{l}_4^0 \cong \mathfrak{l}_{16}$, $\mathfrak{l}_4^1 \cong \mathfrak{l}_5^1$, $\mathfrak{l}_8^{p, q, q}, \mathfrak{l}_8^{p, q, 0}$, $\mathfrak{l}_8^{p, 0, s} \cong \mathfrak{l}_{19}^{p, s}$, $\mathfrak{l}_{11}^{p,q,q,s}, \mathfrak{l}_{11}^{p,q,0,s}, \mathfrak{l}_{11}^{p,0,r,s}$ $\mathfrak{l}_{13} \cong \mathfrak{l}_1^{0, 0}$, $\mathfrak{l}_{14} \cong \mathfrak{l}_2^{1,0} \cong \mathfrak{l}_{20}^0$, $\mathfrak{l}_{15}^p \cong \mathfrak{l}_8^{0, 0, p}$, $\mathfrak{l}_{20}^1 \cong \mathfrak{l}_2^{0, 1}$, $\mathfrak{l}_{22}^{q, 0} \cong \mathfrak{l}_8^{0, 1, 0}$, $\mathfrak{l}_{22}^{q, 1} \cong \mathfrak{l}_8^{0, 1, 1}$, $\mathfrak{l}_{23}^p \cong \mathfrak{l}_8^{1, 0, p}$, $\mathfrak{l}_{25}^{p,p,r} \cong \mathfrak{l}_{11}^{0,p,p,r}$, $\mathfrak{l}_{25}^{p,0,r} \cong \mathfrak{l}_{25}^{0,p,r} \cong \mathfrak{l}_{11}^{0, p, 0, r}$,  whose structure equations  in terms of a $J$-adapted basis $(f_i)$  are given in  Table \ref{fig:LCSKT6}.  In particular, $\mathfrak{g}$ admits a  LCSKT structure  $(J, g)$  if and only if $\mathfrak{g}$ is isomorphic to one of the Lie algebras listed in Table $\ref{fig:LCSKT6}.$ \end{corol}

    \begin{proof} The first part of the corollary follows from the fact that among the  five matrices in Theorem \ref{thm:class6cplx}, only the first three are diagonalizable over $\C$, and the condition on the eigenvalues  of  $A$ imposes that the entries of $A$  have the above expression. Conversely, if  the matrix $A$  associated to $ad_{e_6}\vert_{\mathfrak{n_1}}$ has the above form in some $J$-adapted basis $(e_i)$, then we see that the  metric $g = \sum_{i =1}^6 (e^i)^2$ is Hermitian, because the basis is $J$-adapted. Moreover, $A$  is clearly normal in the $J$-adapted basis, so is $g$-normal. Finally, according to Corollary \ref{cor-LSKT}, we can distinguish whether or not there exists a LCSKT structure by examining the real parts of  the eigenvalues of $A$: if one of the eigenvalues of $A$ has a nonzero real part $\mu$, then there exists an LCSKT structure if and only if $\mu \neq - \frac{a}{2}$. \end{proof}

\begin{remark}  In   Table \ref{fig:LCSKT6} we list the almost abelian Lie algebras admitting a twisted SKT structure  and say whether there exists a SKT or LCSKT structure. Each time, an  explicit example of twisted SKT structure $(J, g)$,   is given by
$$ Jf^1 = f^6, Jf^2=f^3, Jf^4=f^5, \quad g = \sum_{i =1}^6 (f^i)^2.
$$
Note that the Lie algebras    can admit  other complex structures  and the  existence of special types of  Hermitian  metrics strongly depend on the complex structure, as we will see in the next section.
\end{remark}

\section{Compatibility with other types of Hermitian metrics}

Let $\mathfrak{g}$ be an almost abelian Lie algebra of real dimension $2n$ endowed with a complex structure  $J$ and denote by $J_1$ the restriction $J\vert_{\mathfrak{n_1}}$.
In this section, we investigate, fixed  the complex structure $J$,  the interplay of the twisted SKT condition with other types of Hermitian metrics.

As a first result, as  a consequence of  the results in Section \ref{sect3},  we  can  prove  the following

\begin{lemm} 
If  $(\frak g, J)$ admits a  twisted SKT metric $g$, and   $(e_1, \ldots, e_{2n})$ is   a  $g$-orthonormal $J$-adapted basis,  then the set  of all  twisted SKT  $J$-Hermitian   metrics  on  $(\frak g, J)$  coincides with the set $\mathcal G$ of  product metrics $g^{h,u} = h + k^u$, where  $h$ is $J_1$-Hermitian metric on $\frak n_1$ such that $ad_{e_{2n}}\vert_{\mathfrak{n}_1}$ is $h$-normal, $u$ is a non-zero vector  in a complement $\frak c$ of $\frak n_1$ and $k^u$  is the metric  on ${\mbox{span}} \langle u, J u \rangle$ such that  the basis  $(u, Ju)$ is  orthonormal. \end{lemm}

 \begin{proof} Note first  that the  metric $g$ can be written as $g^{h, u} = h _0 + k^u,$ with    $h_0$   the restriction of $g$ to $\frak n_1$ and $u = e_1.$  By Section \ref{sect3}    a twisted SKT  $J$-Hermitian metric   on $(\mathfrak{g}, J)$ is completely determined  by its  restriction to $\mathfrak{n}_1$ and  by a vector $u$  in $ \mathfrak{n}$  which is  orthogonal to $\mathfrak{n}_1$, so  if $g'$ is  a twisted SKT  $J$-Hermitian metric then $g'$ can be written as a metric of the form $g^{h,u}$.  Conversely, given   a product metric of the form $g^{h,u} \in {\mathcal G}$,     
by construction  $g^{h,u}$ is compatible with $J$ and  by  Theorem \ref{thm:caractLCSKT}   it is    twisted SKT.  \end{proof}

\begin{remark}  \label{remAu} Denote by  $(a, v, A)$ the algebraic data  associated to $(\frak g, J,g)$ in the basis $(e_i)$, by $(a^u, v^u,A^u)$ the algebraic data  with respect to a $g^{h,u}$-orthonormal $J$-adapted basis $(e'_i)$, i.e. 
\[[Ju, u] = a^u u + v^u,  \quad A^u = ad_{Ju}\vert_{\mathfrak{n}_1},\]
and by  $P$ the change of basis matrix   from $(e_2, \ldots, e_{2n-1})$ to $(e'_2, \ldots, e'_{2n-1})$.
If we write the vector $u$   as $u = ce_1 + w$  (with respect to the  splitting $\mathfrak{n} = \R e_1 \oplus \mathfrak{n_1}$ induced by the basis $(e_i)),$  then   we obtain the relations
$$
a^u = ca,  \quad A^u = cP^{-1}AP,  \quad Pv^u = c^2 v + c(A - a)w. 
$$
Therefore,  if $P$ is the identity matrix, $a$ and $A$ are only changing by a constant non-zero factor $c$. Note that when $g$ and $g^{h, u}$ coincide on $\mathfrak{n}_1$, we may always assume that this is the case by taking $e'_i = e_i$ for  $i = 1, \ldots 2n - 1$.
\end{remark}

\begin{lemm} \label{pps:kahlerinv}
Let $\mathfrak{g}$ be an almost abelian Lie algebra  endowed with a complex structure  $J$ and admitting  a  twisted SKT  $J$-Hermitian metric. 
\begin{enumerate}

\item[(i)] The condition    $A^u \in {\frak {so}} ({\frak n}_1)$, i.e. $A^u$  being  antisymmetric with respect to the restriction $h = g^{h,u} |_{{\frak n}_1}$,   is either  satisfied  by all  $g^{h,u} \in \mathcal G$ or by none.

As a consequence, $(\frak g, J)$  has a    K\"ahler  metric if and  only if  there exists $g^{h,u}  \in \mathcal{G}$, satisfying both   the conditions $A^u \in \mathfrak{so}({\mathfrak n}_1)$ and $v^u \in {\mbox {Im}} (A^u - a^u \, {\mbox {Id}})$. 

\item[(ii)]  The condition $v^u  \in {\mbox {Im}} (A^u - a^u \, {\mbox{Id}})$ is  either satisfied  by  all $g^{h,u} \in \mathcal G$ or by none.

\end{enumerate}

\end{lemm}

\begin{proof}
We know that  $A^u$ is normal  with respect to $h = g^{h,u}  |_{{\frak n}_1}$, so  by using the spectral theorem and the description of the eigenvalues of $A^u$ from Theorem \ref{thm:caractLCSKT}, we have that   $A^u  \in {\frak {so}} ({\frak n}_1)$ if and only if  ${\mbox {Tr}} (A^u) = 0$.   Since  ${\mbox {Tr}} (A^u) = c {\mbox {Tr}} (A)$,   the first part of  $(i)$  follows.

For  the second part of (i)  recall that by \cite[Lemma 3.6]{fino2021generalized} a metric $g^{h,u} \in \mathcal{G}$ is K\"ahler iff $A^u$ is skew-symmetric  with respect to $h$  and $v^u = 0$.   So  $(\frak g, J)$ is K\"ahler if and only if there exists $g^{h,u}$ such that   $A^u$ is skew-symmetric with respect to $h$  and $v^u = 0$.   We can prove that  this is  equivalent to show that there exists  $u$ such that   $v^u \in Im( A^u  - a {\mbox{Id}})$.  Indeed, let $z$ be such that $v^u = (A^u - a {\mbox {Id}})z$, then for $\tilde u = u  - z$ and $P^u = Id$ on $\mathfrak{n}_1$ we  get $v^{\tilde u} = v^u  - (A^u - a^u Id)z = 0$, so $g^{h, \tilde{u}}$ is K\"ahler.

To  prove  (ii)  it is sufficient to show that, fixed $g \in \mathcal{G}$, $v \in {\mbox {Im}}  (A - aId)$ if and only if  $v^u \in {\mbox {Im}} (A^u - a^u Id)$, for every $g^{h,u } \in\mathcal{G}$.  By Remark  \ref{remAu}, $Im(A^u - a^u Id) = P^{-1} Im(A - aId)$ and $Pv^u = c^2v + (A - a Id)w$, so $v^u \in {\mbox {Im}}  (A^u - a^uId)$ iff $v \in  {\mbox {Im}}  (A - aId)$. Thus  either $v^u$ belongs to ${\mbox {Im}}  (A^u - a^uId)$, for every $g^{h,u } $,  or for none and (ii) follows.
 \end{proof}

By  \cite{Alexandrov, Ferreira}  a   Hermitian metric  $g$ on  an almost abelian Lie algebra  which is  both  twisted SKT and balanced is  K\"ahler.
More in general,  by  \cite[Theorem 3.6]{finoBalanced}    if  $( \frak g,J)$  admits an SKT metric and a balanced metric, then it  has a K\"ahler metric as well. In a similar way we can prove the following

\begin{prop}
Let $\mathfrak{g}$ be an almost abelian Lie algebra equipped with a complex structure $J$. If  $(\frak g, J)$ has  a balanced metric and a twisted SKT  metric, then there exists a K\"ahler metric.
\end{prop}

\begin{proof}
Recall that by  \cite[Theorem 3.6]{finoBalanced}  the expression  of the Lee form   $\theta$ in terms of a $J$-adapted basis  $(e_i)$  is given by
\begin{equation} \label{exprtheta} \theta = - Tr(A) e^{2n} + (Jv)^\flat.\end{equation}
In a similar way as in the proof of Theorem 3.6 in \cite{finoBalanced}  we can show that  the existence of   a balanced metric $b$  and a twisted SKT metric  $g$ implies that  ${\mbox {Tr}}(A) = 0$, since as an endomorphism $A$  depends  only on $J$,  up to a non-zero scalar. Therefore  $A$ is skew-symmetric with respect to $g$ and we can proceed  as in \cite{finoBalanced}   to construct  a K\"ahler metric by combining  $b$ and $g$.

\end{proof}

\begin{remark} \label{remdtheta} Using that  $$d\theta(X, Y) = - \theta([X, Y]), \quad \forall X, Y \in \frak g,$$   and that $\frak n$ is an abelian ideal,  we obtain
$$d\theta = e^{2n} \wedge (J A^T v)^\flat,$$ 
where $A^T$ is the transpose of $A$ with respect to $g$, $(J A^T v)^\flat$ denotes the dual form of $J A^T v$ with respect to the metric $g$ and   $(e_i)$ is an   orthonormal $J$-adapted basis.
Indeed,
  $d\theta (X, Y)= 0$, for every $X, Y   \in \mathfrak{n}$ and  $d(e^{2n}) = 0$.  Therefore, by \eqref{exprtheta} we have 
\[d\theta(e_{2n}, X) =  g(Jv, AX) = (J A^T v)^\flat(X), \quad X \in \frak g.\]
\end{remark}

Locally conformally balanced  (LCB) metrics on almost abelian Lie algebras have been studied in \cite{Paradiso}. Using Remark \ref{remdtheta}  we can prove the following

\begin{prop}\label{prop:lcb} Let $\mathfrak{g}$ be an almost abelian Lie algebra equipped with a complex structure $J$ and admitting a  twisted SKT  $J$-Hermitian metric $g$.  Then the condition
\begin{equation}\label{eq:24NEW}
\{ b v^u + (A^u - a^u) x \mid b \in \R \setminus \{0 \}, x \in \mathfrak n_1 \} \cap  \ker({(A^u)}^T)   \neq \{0\} 
\end{equation}
is either satisfied  by  every $g^{h,u } \in\mathcal{G}$ or  by none. Moreover, there exists a metric $g^{h,u}  \in \mathcal{G}$ which is also LCB if and only if  either $v^u \in {\mbox {Im}} (A^u -a^u {\mbox {Id}})$ or \eqref{eq:24NEW} holds.

\end{prop}

\begin{proof}  The  metric $g$ can be written as  the product $g^{h_0, e_1} = h_0 + k^{e_1}$, where   $h_0 =g \vert_{\frak n_1}$ and $(e_i)$ is a $g$-orthonormal $J$-adapted basis.
To prove the first part of the statement,  it is sufficient to show that, if the condition \eqref{eq:24NEW} is satisfied by $g$,  then it is satisfied  also by every  $g^{h,u} \in \mathcal G$.  We will  first consider the  two particular cases:  $g^{h,  e_1 }$  and $g^{h_0, u}$.   If we   fix  $u = e_1$, we can show that   $\ker((A^{e_1})^T) = P^{-1} \ker(A^T)$, where $(A^{e_1})^T$  and $A^T$ denote respectively  the transpose of $A^{e_1}$ with respect to $h$ and the transpose  of $A$ with respect to $h_0$.   Indeed, by definition, $h$ is a  Hermitian metric on ${\frak n}_1$  such that $A^{e_1}$ is $h$-normal. Denote by $Q$ the Lagrange interpolation polynomial which sends each eigenvalue of $A^{e_1}$ (or equivalently $A$) on its conjugate, as $A^{e_1}$ is $h$-normal, we may simultaneously diagonalize $A^{e_1}$ and $(A^{e_1})^T$, which implies that $(A^{e_1})^T = Q(A^{e_1})$ and for the same reason $A^T = Q(A)$. But then $P v^{e_1} = v$, and $\ker((A^{e_1})^T) = \ker(Q(A^{e_1})) = \ker(P^{-1}Q(A)P) = P^{-1}\ker(A^T)$. From the fact that ${\mbox{Im}}(A^{e_1} - a^{e_1}{\mbox{Id}}) = ¨P^{-1} {\mbox{Im}}(A - a{\mbox{Id}})$, we see that as $P$ is invertible the first part of the statement is true in this particular case.
Secondly, if   we consider $g^{u, h_0}$, i.e. if we fix $h_0$,  we may assume that $P = {\mbox{Id}}$. Then, $b v^u + (A^u - a^u{\mbox{Id}})x = b c^2 v + (A - a \, {\mbox{Id}}) (b c w + c x)$,  for every  $b \in \R\backslash\{0\},$   and for every  $x \in \mathfrak{n}_1$, and the first part of the statement follows  also in this other case. 
Now, let $g^{h, u }= h + k^u$ be any element of $\mathcal{G}$, then by considering the intermediary metric $g^{h, e_1}$,  we have  that the condition \eqref{eq:24NEW}
is satisfied by $g = g^{h_0, e_1}$ if and only if  \eqref{eq:24NEW} holds for every  $g^{h, e_1} \in \mathcal G$ or equivalently  if and only  \eqref{eq:24NEW}  it is satisfied by $g^{h, u}$ and the first part of the statement follows.

For the second part we  use    that   a metric  $g^{h,u} \in \mathcal{G}$  is LCB if and only if ${(A^u)}^T v ^u= 0$ (see Remark  \ref{remdtheta}).
To construct a  metric $g^{h,u} \in \mathcal G$ such that  $ {(A^u)}^T v ^u = 0$ we can proceed in  the following way. Let  $g = g^{h_0, e_1} \in \mathcal G$. If $v \in {\mbox{Im}} (A - a {\mbox{Id}})$ (see  the proof of Lemma \ref{pps:kahlerinv}), then there exists $u$ such that for $h = h_0$, $ (A^u)^T v^u = 0$ and $g^{h_0,u}$ is LCB.  If $v \notin {\mbox{Im}}(A - a{\mbox{Id}})$, then for $u \in \mathfrak{n}$, taking $g^u = g$ on $\mathfrak{n}_1$ and $P = {\mbox{Id}}$, since $v^u = c^2 v + c (A - a Id) w$, $ {(A^u)}^T v ^u = 0$ if and only if $$
\{ b v + (A - a) x \mid b \in \R \setminus \{0 \}, x \in \mathfrak n_1 \} \cap  \ker({(A^u)}^T)   \neq \{0\}. \label{eq:24}
$$ Thus, if the above condition is satisfied then we can find $u$ such that $g^{h_0,u}$  is LCB. Conversely, if $g^{h,u}$ is LCB then the above condition is satisfied by $g^{h_0,u}$, so by $g$ according to the first part of the statement.
\end{proof}

This allows us to construct an example of   unimodular non-K\"ahler  almost abelian  Lie algebra admitting a  Hermitian structure which is both  LCSKT  and  LCB. Recall that a Lie algebra  $\frak g$ is unimodular if ${\mbox{Tr}} (ad_x) =0$, for every $x \in \frak g$. By \cite{Milnor} this is  necessary condition for the associated Lie group to admit lattices.

\begin{ex} {Consider the $8$-dimensional almost abelian Lie algebra $\frak g(a, v, A)$, with $a = 1, v = e_6$, $$  A = 
\left (\begin{array}{cccccc}
r & -s & 0 & 0 & 0 & 0\\
s & r & 0 & 0 & 0 & 0\\
0 & 0 & r & s' & 0 & 0 \\
0 & 0 & -s' & r & 0 & 0 \\
0 & 0 & 0 & 0 & 0 & 0 \\
0 & 0 & 0 & 0 & 0 & 0 
\end{array} \right),
$$
the  complex structure $J$ defined by $Je_1 = e_8, Je_2 = e_3, Je_4 = e_5, Je_6 = e_7$ and  $\frak n_1 = {\mbox{span}} \langle e_2, e_3, e_4, e_5, e_6, e_7 \rangle$. Note that if $r = -\frac{1}{4}$  the Lie algebra is unimodular. Since  $v \in Im(A-a \, {\mbox{Id}})$,   the metric $g =\sum_{i}^8  (e^i)^2$ is both  LCSKT  and LCB. 
Moreover,  every  twisted SKT metric is  of the form   $g^{h,u}= h + k^u$,  where  $u = ce_1 + w \in \mathfrak{n}\backslash\mathfrak{n}_1$,   $A^u = ad_{Ju}$  is  $h$-normal. and  $h$  is an arbitrary $J_1$-hermitian metric  if $(r,s) =(0,0)$ and
$$h = t((e^2)^2 + (e^3)^2) + u((e^4)^2 + (e^5)^2) + x(e^6)^2 + 2y e^6 e^7 + z(e^7)^2,  \quad t, u > 0,  \, xz > y^2, \, x + z \geq 0$$
if $(r,s) \neq (0,0)$.
Among these metrics, those which also are LCB are precisely the metrics such that $(A^u)^T v^u = 0$, where the transpose is with respect to $g^u$, or equivalently such that   $w \in {\mbox {span}}  \langle e_6, e_7 \rangle$.}
Note that for $r \neq 0$,  the Lie algebra does not admit any K\"ahler metric  since the eigenvalues of $A$ are not imaginary.
  \end{ex}

We investigate now the existence of metrics which are both twisted SKT and Bismut Ricci flat. Recall that  the Ricci form  $\rho^B$ of the Bismut connection is given in general by 
\[\rho^B(X, Y) = \frac{1}{2} \sum_{i=1}^{2n} g(R^B(X, Y)e_i, Je_i) \]
where $(e_i)$ is any orthonormal basis.  For  a Hermitian almost abelian  Lie algebra $(\frak g (a, v, A), J, g)$  $\rho^B$ has the following expression
\[\rho^B = - (a^2 - \frac{1}{2}a Tr(A) + \Vert v \Vert^2) e^1 \wedge e^{2n} - (A^Tv)^\flat \wedge e^{2n},\]
where  $(e_i)$ is an orthonormal $J$-adapted basis and $x^\flat = g(x, \cdot)$
(see \cite[Prop. 4.8]{Arroyo_2019}).
Therefore, since $A^Tv$ and $e^1$ are orthogonal, the vanishing of $\rho^B$ is equivalent to the conditions
\[ a^2 - \frac{1}{2}a \, {\mbox {Tr(}}A) = -\Vert v \Vert^2,  \quad A^T v = 0.\]

\begin{remark} \label{remrhob=0}
Note that by Remark \ref{remAu}     we have  \[(a^u)^2 - \frac{1}{2} a^u Tr(A^u) = c^2(a^2 - \frac{1}{2}aTr(A)), \quad \forall g^{h,u} \in \mathcal G \]
and so if $\rho^B=0$  either 
  $(a^u)^2 - \frac{1}{2}a^uTr(A^u)=0$  for every $g^{h, u} \in \mathcal G$  or   $(a^u)^2 - \frac{1}{2}a^uTr(A^u)<0$  for every $g^{h,u} \in \mathcal{G}$.
\end{remark}

\begin{remark}
Let $g$ be a twisted SKT metric. By the  spectral theorem and the condition  ${\mbox {Re}} ( {\rm {Spec}} (A)) \subset \{0, \mu\}$,  it follows  that   $A \in \mathfrak{so(n_1)}$ iff ${\mbox {Tr}}(A) = 0$.  Therefore, if ${\mbox {Tr}} (A) = 0$,  the vanishing of  $\rho^B$ implies  $a^2 + \Vert v \Vert^2 = 0$,  and consequently that   $(J, g)$ has to be K\"ahler.
\end{remark}

\begin{prop} \label{pps:brflatinv}  Let $\frak g$ be an almost abelian Lie algebra endowed with a complex structure $J$.  Then there exists  a  Bismut Ricci flat  twisted   SKT metric  iff either 
\begin{equation} \label{cond1rhoB} a^2 - \frac{1}{2} a  {\mbox {Tr}} (A) = 0, \quad  v  \in  {\mbox {Im}}  (A -a  \,  {\mbox {Id}}) \end{equation}
or
\begin{equation} \label{cond2rhoB} a^2 - \frac{1}{2} a  {\mbox {Tr}} (A) < 0,  \quad \{ \lambda v + (A - a) x \mid \lambda \in \R \setminus \{0 \}, x \in \mathfrak n_1 \} \cap  \ker (A^T)   \neq \{0\}. \end{equation}

In particular,  every  Bismut Ricci flat  LCSKT metric is LCB. 
\end{prop}

\begin{proof}   A  twisted SKT metric  has  $\rho^B=0$  if and only if  
\[ a^2 - \frac{1}{2}a \, {\mbox {Tr}} (A) = -\Vert v \Vert^2,  \quad A^T v = 0.\]
If $a^2 - \frac{1}{2}aTr(A) = 0$, it follows that $v = 0$, so $v\in Im(A -a {\mbox{Id}})$. Conversely, if the conditions \eqref{cond1rhoB}  are satisfied for a metric $g \in \mathcal{G}$, it follows from the proof of Lemma \ref{pps:kahlerinv} that there exists $g^{h, u} \in \mathcal{G}$ such that $v^u = 0$. Then by Remark  \ref{remrhob=0}  we get $$(a^u)^2 - \frac{1}{2}a^u {\mbox {Tr}} (A^u) ) = c^2 \left (a^2 - \frac{1}{2}a {\mbox {Tr}} (A) \right)= 0 = \Vert v^u \Vert^2$$ and  so $g^{h, u}$ is twisted SKT and Bismut Ricci flat.

If $a^2 - \frac{1}{2}a Tr(A) < 0$,  by imposing that $A^Tv = 0$ and $v \neq 0$, we have   ${\mbox {span}} \langle v \rangle \cap \ker(A^T) \neq \{0\}$. Conversely, if the conditions \eqref{cond2rhoB} are satisfied for a metric $g \in \mathcal{G}$, it follows from Proposition 4.2 that there exists $g^{h,u} \in \mathcal{G}$ such that $(A^u)^T v^u = 0$. Let  $\tilde u := s u =  c s e_1 +  s w$,  with  $s > 0$.  Then $g^{h, \tilde{u}}$ satisfies
$$a^{\tilde u} = s a^u, v^{\tilde u} = s^2 v^u,A^{\tilde u} = s A^u.$$
As a consequence  $(A^{\tilde{u}})^T v^{\tilde u} = 0$  and 
$$(a^{\tilde u})^2 - \frac{1}{2} a^{\tilde u} {\mbox {Tr}} (A^{\tilde u}) = s^2 \left ( (a^u)^2 - \frac{1}{2} a^u {\mbox {Tr}}(A^u)\right ) < 0, \quad 
\Vert v^{\tilde u} \Vert^2 = s^4 \Vert v^u \Vert^2 > 0,$$
so that for 
$$s^2 = - \frac{(a^u)^2 - \frac{1}{2} a^u {\mbox {Tr}}(A^u)}{\Vert v^u \Vert^2} > 0,$$
we have
$$(a^{\tilde u})^2 - \frac{1}{2} a^{\tilde u} \, {\mbox {Tr}} (A^{\tilde u}) = - \Vert v^{\tilde u} \Vert^2.$$
It follows that for this $s$, $g^{h,\tilde{u}}$  is twisted SKT and Bismut-Ricci flat. The last part of the statement follows then from Proposition \ref{prop:lcb}.
\end{proof}

We can prove that a  twisted SKT Bismut-Ricci flat  metric on  an unimodular almost abelian Lie algebra  is  K\"ahler flat.
\begin{corol} 
Let $\mathfrak{g}$ be   a unimodular almost abelian Lie algebra admitting a twisted SKT structure $(J, g)$.   Then $(\frak g, J)$   has  a twisted SKT Bismut-Ricci flat metric  if and only if  $J$  is   bi-invariant, i.e. $[Jx, y] = J[x, y]$, for every $x, y \in \frak g$. Moreover, every twisted Bismut-Ricci flat metric has to be  K\"ahler flat.\end{corol}
\begin{proof}  By the    unimodularity of $\frak g$ it follows that  ${\mbox {Tr}}(B) = 0 = a +  {\mbox {Tr}} A)$,  for every $J$-adapted  basis. Moreover,  since $(J, g)$ is twisted SKT we have  $a^2 + \frac{1}{2}a {\mbox {Tr}}  (A) + \Vert v \Vert^2 = \frac{1}{2}a^2 + \Vert v \Vert^2 \geq 0$. 

The existence of a twisted SKT Bismut-Ricci flat metric $g'$  implies that $\frac{1}{2}a^2 + \Vert v \Vert^2 \leq 0$, so $a = 0$ and $v = 0$. As a consequence ${\mbox {Tr}} (A) = 0$ and  the metric $g'$  is K\"ahler.  Let $(e_i)$ be a $g'$-orthonormal $J$-adpted basis.  Note that $J$ is bi-invariant   if and only if  $[B, J] =0$. Therefore, since $[A, J_1] =0$,   it remains to check that $BJ$ and $JB$ coincide on $e_1$ and $e_{2n}$, which follows by $$Be_1 =ae_1 + v = 0 = JBe_{2n}.$$

Conversely, if  $J$ is bi-invariant and  $(e_i)$  is a $g$-orthonormal $J$-adapted basis, then $ae_1 + v = Be_1 = JBe_{2n} = 0$, so $a = 0$ and $v = 0$. It follows that $a^2 - \frac{1}{2}aTr(A) + \Vert v \Vert^2 = 0$ and $v \in Im(A-aId)$, so according to Proposition  \ref{pps:brflatinv} there exists a Bismut-Ricci flat metric $g^{h, u}$. As $\mathfrak{g}$ is unimodular, we have  $\frac{1}{2}{(a^u)}^2 + \Vert  v^u \Vert^2 = 0$, so $ a^u = 0$,
 $v^u= 0$ and finally ${\mbox {Tr}} (A^u) = 0$, so $A^u\in \mathfrak{so}(g^{h, u} \vert_{\mathfrak{n_1}})$ and $g^{h, u}$ is K\"ahler.  Moreover, the metric $g^{h, u}$ has to be  flat since  Ricci flatness implies flatness (see  \cite{AK}).
\end{proof}

We can use Proposition  \ref{pps:brflatinv}  to construct an example of a non-K\"ahler  non-unimodular  almost abelian Lie algebra  admitting  a metric which is both LCSKT and Bismut-Ricci flat.  

\begin{ex} 
Consider the $6$-dimensional  Lie algebra $\frak g :=\frak g(a, v, A)$, with $a = 1, v = e_4$, \[A = 
\begin{pmatrix}
r & -s & 0 & 0 \\
s & r & 0 & 0 \\
0 & 0 & 0 & 0 \\
0 & 0 & 0 & 0
\end{pmatrix}.
\]
and the complex structure $J$ defined by $Je_1 = e_6, Je_2 = e_3, Je_4 = e_5$. 
Then $a^2 - \frac{1}{2}a {\mbox {Tr}} (A) = 1 - r$, $A^Tv = 0$ and $v \in Im(A-a)$.
Therefore $(\frak g, J)$  has  a Bismut-Ricci flat  LCSKT metric  iff $r \geq 1$. Moreover, for $r \geq 1$,  the Lie algebra is non-K\"ahler since  $A$ is not skew-symmetric.
\end{ex}

\begin{ex}
The existence of a twisted SKT Bismut-Ricci flat metric  depends on the complex structure. Indeed, if we consider  the  Lie algebra $\frak g(a, v, A)$ with 
$a = 1, v = e_4,$ 
\[ A = 
\begin{pmatrix}
2 & 0 & 0 & 0 \\
0 & 2 & 0 & 0 \\
0 & 0 & 0 & 0 \\
0 & 0 & 0 & 0
\end{pmatrix} \]
and the  complex structure $Je_1 = e_6, Je_2 = e_3, Je_4 = e_5$. Then the metric $g = \sum_{i = 1}^6 (e^i)^2$ is twisted SKT and Bismut-Ricci flat. However,  if we consider  the new complex structure $J'$ defined by 
 $$J'(e_1 + e_4) = e_6, \,  J'e_2 = e_3,  \, J'e_4=e_5,$$  we have that  the metric,  such that the basis $(e_1 + e_4,  e_2, e_3, e_4, e_5, e_6)$ is orthonormal,  is LCSKT,  but   the Lie algebra does not admit any $J'$-Hermitian  twisted SKT Bismut-Ricci flat metric.
\end{ex}

\section*{Appendix}

\begin{table}[h]
   \centering

\begin{tabular}{|c|c|c|c|}
\hline
Conditions on $(a, v, A)$                        & Possible $\alpha$                                   & K\"ahler & SKT                   \\ \hline
$a = 0$, $\{\mu\} \subset Re \,  {\rm {Spec}}(A) \subset \{0, \mu\}$ with $\mu \neq 0$    & $- 2 \mu e^{2n} + s e^1,  \, \, s \in \R$                           & No     & No  \\ \hline
$a \neq 0$, $\{\mu\} \subset Re  \, {\rm {Spec}} (A) \subset \{0, \mu\}$ with $\mu \neq 0$ & $(-a- 2 \mu) e^{2n}$                      & No     & $\mu = - \frac{a}{2}$ \\ \hline
$a = 0$, $v = 0$, $Re \,  {\rm {Spec}} (A) = \{0\}$                  & $s e^{2n} + t e^1 + \beta, \, \,  s, t \in \R, \beta \in  ({\rm {Im}} (A)^\perp)^*$              & Yes    & Yes                   \\ \hline
$a = 0$, $v \neq 0$, $Re \, {\rm {Spec}} (A) = \{0\}$               & $s  e^{2n} +  \beta,  \,  \, s \in \R, \beta \in ( {\rm {Im}} (A)^\perp)^*$                       & No     & Yes                   \\ \hline
$a \neq 0$, $v = 0$, $Re \,  {\rm {Spec}}(A) = \{0\}$               & $s e^{2n} + t e^1, \,  \, s,t \in \R$                                & Yes    & Yes                   \\ \hline
$a \neq 0$, $v \neq 0$, $Re   \, {\rm {Spec}}(A) = \{0\}$            & $s e^{2n} + t  (\Vert v \Vert^2 e^1 - a v^\flat), \, \,  s,t \in \R$ & No     & Yes                   \\ \hline
\end{tabular}
    \caption{List of conditions on  $(a, v, A)$ and possible $\alpha$} \label{fig:alphas} 
  \end{table}

\newpage

\begin{table}[htbp]
\centering
\begin{tabular}{|c|c|c|c|c|}
\hline
Name                                                                                               & Structure equations                                                                & K\"ahler  & SKT                          & LCSKT \\ \hline
$\mathfrak{l_1}^{p, p}$                                                                            & $(f^{16},  pf^{26}, pf^{36}, pf^{46}, pf^{56}, 0)$                                 & $p = 0$ & $p \in \{0, - \frac{1}{2}\}$ & $p \neq - \frac{1}{2}$              \\ \hline
$\mathfrak{l}_{17}^p$                    & $(f^{16}, pf^{26}, pf^{36}, 0, 0, 0)$                                              & $p = 0$ & $p \in \{0, - \frac{1}{2}\}$ & $p\neq -\frac{1}{2}$                          \\ \hline
$\mathfrak{l}_{18}^q$                    & $(qf^{16}, f^{26}, f^{36}, 0, 0, 0)$                                               & \_      & $q = -2$                     & $q \neq -2$               \\ \hline
$\mathfrak{l}_2^{q, 1}$                                               & $(qf^{16}, f^{26}, f^{36}, f^{46}, f^{56}, 0)$                                     & \_      & $q = -2$                     & $q \neq -2$               \\ \hline
$\mathfrak{l}_{16}$                                                        & $(f^{16}, f^{26}+f^{16}, f^{36}, 0, 0, 0)$                                         & \_      & \_                           & $\checkmark$                             \\ \hline
$\mathfrak{l}_4^1$                                                         & $(f^{16}, f^{26}+ f^{16}, f^{36}, f^{46}, f^{56}, 0)$                              & \_      & \_                           & $\checkmark$                             \\  \hline
$\mathfrak{l}_8^{p, q, q}$                                                                         & $(pf^{16}, qf^{26}, qf^{36}, qf^{46} + f^{56}, -f^{46} + qf^{56}, 0)$              & $q=0$   & $q \in \{0, - \frac{p}{2}\}$ & $q\neq -\frac{p}{2}$               \\ \hline
$\mathfrak{l}_8^{p, q, 0}$                                                                         & $(pf^{16}, qf^{26}, qf^{36}, f^{56}, -f^{46}, 0)$                                  & $q=0$   & $q\in \{0, -\frac{p}{2}\}$   & $q \neq - \frac{p}{2}$                \\ \hline
$\mathfrak{l}_{19}^{p, s}$                                         & $(pf^{16}, 0, 0, sf^{46} + f^{56}, -f^{46} + sf^{56}, 0)$                          & $s = 0$ & $s \in \{0, -\frac{p}{2}\}$  & $s\neq -\frac{p}{2}$                 \\ \hline
$\mathfrak{l}_{11}^{p,q,q,s}$                                                                      & $(pf^{16}, qf^{26}+f^{36}, -f^{26}+qf^{36}, qf^{46}+sf^{56}, -sf^{46}+qf^{56}, 0)$ & $q = 0$ & $q\in \{0, -\frac{p}{2}\}$   & $q\neq -\frac{p}{2}$               \\ \hline
$\mathfrak{l}_{11}^{p, q, 0, s}$                                                                   & $(pf^{16}, qf^{26}+f^{36}, -f^{26}+qf^{36}, sf^{56}, -sf^{46}, 0)$                 & $q=0$   & $q\in \{0, -\frac{p}{2}\}$   & $q\neq -\frac{p}{2}$                 \\ \hline
$\mathfrak{l}_{11}^{p, 0, r, s}$                                                                   & $(pf^{16}, f^{36}, -f^{26}, rf^{46}+sf^{56}, -sf^{46}+rf^{56}, 0)$                 & $r=0$   & $r\in\{0, -\frac{p}{2}\}$    & $r\neq -\frac{p}{2}$                 \\ \hline
$\mathfrak{l}_{13}$                                                   & $(f^{16}, 0, 0, 0, 0, 0)$                                                          & $\checkmark$     & $\checkmark$                          & $\checkmark$                             \\ \hline
$\mathfrak{l}_{14}$                         & $(0, f^{26}, f^{36}, 0, 0)$                                                        & \_      & \_                           & $\checkmark$                            \\ 
$\mathfrak{l}_{15}^p$                                              & $(0, pf^{26}+f^{36}, -f^{26}+pf^{36}, 0, 0, 0)$                                    & $p=0$   & $p=0$                        & $\checkmark$                            \\ \hline
$\mathfrak{l}_{20}^1$                                                 & $(0, f^{26}, f^{36}, f^{46}, f^{56}, 0)$                                           & \_      & \_                           & $\checkmark$                            \\ \hline
$\mathfrak{l}_{22}^{q, 0}$                                         & $(0, f^{26}, f^{36}, f^{56}, -f^{46}, 0)$                                          & \_      & \_                           & $\checkmark$                            \\ \hline
$\mathfrak{l}_{22}^{q, 1}$                                         & $(0, f^{26}, f^{36}, f^{46}+f^{56}, -f^{46}+f^{56}, 0)$                            & \_      & \_                           & $\checkmark$                            \\ \hline
$\mathfrak{l}_{23}^p$                                              & $(f^{16}, 0, 0, pf^{46}+f^{56}, -f^{46}+pf^{56}, 0)$                               & $p=0$   & $p \in \{0, -\frac{1}{2}\}$  & $p\neq -\frac{1}{2}$                          \\ \hline
$\mathfrak{l}_{25}^{p,p,r}$                                     & $(0, pf^{26}+f^{36}, -f^{26}+pf^{36}, pf^{46}+rf^{56}, -rf^{46}+pf^{56}, 0)$       & $p=0$   & $p=0$                        & $\checkmark$                          \\ \hline
$\mathfrak{l}_{25}^{p,0,r}$ & $(0, pf^{26} + f^{36}, -f^{26} +pf^{36}, rf^{56}, -rf^{46}, 0)$                    & $p=0$   & $p=0$                        & $\checkmark$                          \\ \hline
\end{tabular}
 \caption{List of almost abelian  Lie algebras admitting a  twisted SKT structure} \label{fig:LCSKT6}
   
\end{table}

\begin{table}[htbp]
\centering
\begin{tabular}{|c|c|c|c|c|}
\hline
Name                                                                                               & Structure equations                                                                & Unimodular  & LCB                          & Bismut-Ricci flat \\ \hline
$\mathfrak{l_1}^{p, p}$                                                                            & $(f^{16},  pf^{26}, pf^{36}, pf^{46}, pf^{56}, 0)$                                 & $p = -\frac{1}{4}$ & $\checkmark$ & $p = \frac{1}{2}$              \\ \hline
$\mathfrak{l}_{17}^p$                    & $(f^{16}, pf^{26}, pf^{36}, 0, 0, 0)$                                              & $p = -\frac{1}{2}$ & $\checkmark$ & $p=1$                          \\ \hline
$\mathfrak{l}_{18}^q$                    & $(qf^{16}, f^{26}, f^{36}, 0, 0, 0)$                                               & $q = -2$      & $\checkmark$                     & $q \in \{0, 1\}$               \\  \hline
$\mathfrak{l}_2^{q, 1}$                                               & $(qf^{16}, f^{26}, f^{36}, f^{46}, f^{56}, 0)$                                     & $q = -4$      & $\checkmark$                     & $q \in \{0, 2\}$               \\ \hline
$\mathfrak{l}_{16}$                                                        & $(f^{16}, f^{26}+f^{16}, f^{36}, 0, 0, 0)$                                         & $q = -4$      & \_                           & \_                             \\ \hline
$\mathfrak{l}_4^1$                                                         & $(f^{16}, f^{26}+ f^{16}, f^{36}, f^{46}, f^{56}, 0)$                              & \_      & \_                           & \_                             \\  \hline
$\mathfrak{l}_8^{p, q, q}$                                                                         & $(pf^{16}, qf^{26}, qf^{36}, qf^{46} + f^{56}, -f^{46} + qf^{56}, 0)$              & $p + 4q = 0$   & $\checkmark$ & $p \in \{0, 2q\}$               \\ \hline
$\mathfrak{l}_8^{p, q, 0}$                                                                         & $(pf^{16}, qf^{26}, qf^{36}, f^{56}, -f^{46}, 0)$                                  & $p+2q=0$   & $\checkmark$   & $p\in \{0, q\}$                \\ \hline
$\mathfrak{l}_{19}^{p, s}$                                         & $(pf^{16}, 0, 0, sf^{46} + f^{56}, -f^{46} + sf^{56}, 0)$                          & $p+2s = 0$ & $\checkmark$  & $p\in \{0, s\}$                 \\ \hline
$\mathfrak{l}_{11}^{p,q,q,s}$                                                                      & $(pf^{16}, qf^{26}+f^{36}, -f^{26}+qf^{36}, qf^{46}+sf^{56}, -sf^{46}+qf^{56}, 0)$ & $p+4q = 0$ & $\checkmark$   & $p\in\{0, 2q\}$               \\ \hline
$\mathfrak{l}_{11}^{p, q, 0, s}$                                                                   & $(pf^{16}, qf^{26}+f^{36}, -f^{26}+qf^{36}, sf^{56}, -sf^{46}, 0)$                 & $p+2q=0$   & $\checkmark$   & $p\in\{0, q\}$                 \\ \hline
$\mathfrak{l}_{11}^{p, 0, r, s}$                                                                   & $(pf^{16}, f^{36}, -f^{26}, rf^{46}+sf^{56}, -sf^{46}+rf^{56}, 0)$                 & $p+2r=0$   & $\checkmark$    & $p\in\{0,r\}$                 \\ \hline
$\mathfrak{l}_{13}$                                                   & $(f^{16}, 0, 0, 0, 0, 0)$                                                          & \_     & $\checkmark$                          & \_                            \\ \hline
$\mathfrak{l}_{14}$                         & $(0, f^{26}, f^{36}, 0, 0)$                                                        & \_      & $\checkmark$                           & $\checkmark$                            \\ \hline
$\mathfrak{l}_{15}^p$                                              & $(0, pf^{26}+f^{36}, -f^{26}+pf^{36}, 0, 0, 0)$                                    & $p=0$   & $\checkmark$                        & $\checkmark$                            \\ \hline
$\mathfrak{l}_{20}^1$                                                 & $(0, f^{26}, f^{36}, f^{46}, f^{56}, 0)$                                           & \_      & $\checkmark$                           & $\checkmark$                            \\ \hline
$\mathfrak{l}_{22}^{q, 0}$                                         & $(0, f^{26}, f^{36}, f^{56}, -f^{46}, 0)$                                          & \_      & $\checkmark$                           & $\checkmark$                            \\ \hline
$\mathfrak{l}_{22}^{q, 1}$                                         & $(0, f^{26}, f^{36}, f^{46}+f^{56}, -f^{46}+f^{56}, 0)$                            & \_      & $\checkmark$                           & $\checkmark$                            \\ \hline
$\mathfrak{l}_{23}^p$                                              & $(f^{16}, 0, 0, pf^{46}+f^{56}, -f^{46}+pf^{56}, 0)$                               & $p=-\frac{1}{2}$   & $\checkmark$  & $p=1$                          \\ \hline
$\mathfrak{l}_{25}^{p,p,r}$                                     & $(0, pf^{26}+f^{36}, -f^{26}+pf^{36}, pf^{46}+rf^{56}, -rf^{46}+pf^{56}, 0)$       & $p=0$   & $\checkmark$                        & $p = 0$                          \\\hline
$\mathfrak{l}_{25}^{p,0,r}$ & $(0, pf^{26} + f^{36}, -f^{26} +pf^{36}, rf^{56}, -rf^{46}, 0)$                    & $p=0$   & $\checkmark$                        & $p = 0$                          \\ \hline
\end{tabular}
\caption{LCB and Bismut-Ricci flat metrics} \label{fig:LCSKT7}

\end{table}

 \end{document}